\documentclass[10pt]{amsart}

\usepackage{color,graphicx,shortvrb}
\usepackage[latin 1]{inputenc}
\usepackage[colorlinks, linkcolor=blue, citecolor=blue, urlcolor=blue,
breaklinks=true]{hyperref}

\usepackage[dvipsnames]{xcolor}

\usepackage{enumerate}

\usepackage[sort]{natbib}
\usepackage{url}

\newtheorem{theorem}{Theorem}[section]

\theoremstyle{definition}

\theoremstyle{remark}

\def\RR{{\mathbb{R}}}

\def\11{\textbf{$1$}}

\def\ir{\displaystyle{\int_{\RR^N}}}

\DeclareGraphicsExtensions{.jpg,.pdf,.png,.eps}

\def\:{\colon}

\def\R{\mathbb{R}}

\def \ir3 {\int_{\R^3}}

\def\ir{\int_{\R}}

\usepackage{enumerate}

\title{Nonexistence result for a semilinear elliptic problem}
\author{Salvador López-Martínez}
\address{Departamento de Análisis Matemático, Universidad de Granada, Facultad de Ciencias, AvenidaFuentenueva s/n, 18071, Granada, Spain}
\email{salvadorlopez@ugr.es}
\author{Alexis Molino}
\address{Departamento de Análisis Matemático, Universidad de Granada, Facultad de Ciencias, AvenidaFuentenueva s/n, 18071, Granada, Spain}
\email{amolino@ugr.es}
\begin{document}

\begin{abstract} In this paper we prove the nonexistence of nontrivial solution to 
\begin{equation*}
\begin{cases}
-\Delta u =f(u) &\text{in }\Omega,
\\
u=0 &\text{on } \partial \Omega,
\end{cases}
\end{equation*}
being $\Omega \subset \mathbb{R}^N$ ($N\in\mathbb{N}$) a bounded domain and $f$ locally Lispchitz with non-positive primitive.
\end{abstract}
\maketitle
\section{Introduction}
Problems of partial differential equations are extensively studied at present, mainly motivated  by their applications in fields of physics, biology and engineering among others. One of the simplest models of nonlinear elliptic differential equations is the following
\begin{equation}\label{eq10}
\begin{cases}
-\Delta u =f(u) &\text{in }\Omega,
\\
u=0 &\text{on } \partial \Omega,
\end{cases}
\tag{$P$}
\end{equation}
being $\Omega \subset \mathbb{R}^N$ ($N\in\mathbb{N}$) a bounded domain  with boundary of class $\mathcal{C}^{1,1}$ and the source term $f:\mathbb{R}\to\mathbb{R}$ a locally Lipschitz function.


  Along this note, a \emph{classical solution} to \eqref{eq10} (\emph{solution} from now on) will be a function $u\in\mathcal{C}^2(\Omega)\cap\mathcal{C}^{1,\alpha}(\overline{\Omega})$, for some $\alpha\in (0,1)$, satisfying \eqref{eq10} pointwise. Observe that, by regularity results, every bounded weak solution is a solution to this problem (see e.g. \cite{MR2431434}).

%
  
When studying any kind of problem involving differential equations, it is always useful to know necessary conditions for the existence of solution. For instance, it follows immediately that a necessary condition for the existence of a solution $u$ to \eqref{eq10} is that $u$ must satisfy the equality
\begin{equation}\label{1000}
\int_\Omega|\nabla u|^2=\int_\Omega f(u)u.
\end{equation}
In consequence, a straightforward nonexistence result for problem \eqref{eq10} states that if 
\begin{equation}\label{NonEx1}
\text{\emph{ $f(s)s\leq 0$, \quad for all $s\in\mathbb{R}$,}}
\end{equation}
there exists no nontrivial solution to \eqref{eq10}. In addition, the well-known Pohozaev identity (\cite{MR0192184}) yields a sort of generalization of this simple result. To be more precise, every solution $u$ to \eqref{eq10} must satisfy the following equality:
 \begin{equation}\label{Poho}
\frac{1}{2}\int_{\partial \Omega}|\nabla u(x)|^2\,x\cdot \nu(x)dx+\frac{N-2}{2}\int_{\Omega}|\nabla u(x)|^2dx=N\int_\Omega F(u(x))dx,
\end{equation}
where $F(s)=\int_0^s f(t)dt$ for all $s\in\mathbb{R}$ and $\nu$ denotes the unit outward normal to $\partial \Omega$ vector. Observe that if $\Omega$ is starshaped with respect to $0$ (i.e., $x\cdot \nu(x)>0$ on $\partial \Omega$) and $N\geq 3$, the left hand side of \eqref{Poho} is non-negative. Therefore, if 
\begin{equation}\label{NonEx2}
\begin{array}{cc}
\text{\emph{$F(s)\leq 0$, \quad for all $s\in\mathbb{R}$, }}
\end{array}
\end{equation}
 there exists no nontrivial solution to \eqref{eq10} whenever $\Omega$ is starshaped. Keep in mind that condition \eqref{NonEx2} implies that $f(0)=0$. Thus, zero is always a solution.
 
   Condition $sf(s)\leq 0$ clearly guarantees $F(s)\leq 0$, but not conversely. A simple example is $f(s)=\lambda \sin s$, being $\lambda <0$. Where, to our knowledge, the existence of a nontrivial solution until now is unknown. Instead, existence of solutions for $\lambda>0$ were established in \cite{Figue}. In this way, a natural question is whether the condition $\Omega$ is starshaped is essential for the nonexistence of nontrivial solution to \eqref{eq10}, for any bounded domain $\Omega$ and $f$ satisfying \eqref{NonEx2}.
   
      A similar situation arises when one analyzes the well-known supercritical case result, also derived from \eqref{Poho}. Concretely, if $f(s)=\lambda|s|^{p-2}s$, for $\lambda>0$ and $p\geq 2^*$, there exists no nontrivial solution to \eqref{eq10} provided $N\geq 3$ and $\Omega$ is starshaped. However,  it is surprising the existence of nontrivial solutions for $p\geq 2^*$ when the domain is not starshaped. For instance, positive solutions have been found when the domain is an annulus (see the seminal paper \cite{MR0477445} and references therein) or for domains with small holes (\cite{DELPINO2002511}).

But nevertheless, much less is known about the influence of the geometry of $\Omega$ in the existence of solution to problem \eqref{eq10} in the case $F(s)\leq 0$ and the literature contains only partial nonexistence results. Observe that for functions $f$  globally Lipschitz, with  $L-$Lipschitz constant, it follows that  $|f(s)|\leq L |s|$. Thus, applying Poincaré inequality in \eqref{1000}, we obtain
\begin{equation*}
\lambda_1\int_\Omega u^2\leq  \int_\Omega|\nabla u|^2=\int_\Omega f(u)u\leq L \int_\Omega u^2.
\end{equation*}
Therefore, this simple computation gives the nonexistence of nontrivial solutions as long as $L<\lambda_1$, being $\lambda_1$ the first eingenvalue for the Laplacian operator with zero Dirichlet boundary conditions. In this line, in \cite{RICCERI20082964} and \cite{MR2536819} the authors prove the nonexistence provided that $L\leq 3\lambda_1$ ($N\geq 2$). Recently, in \cite{Goubet}, the nonexistence of nontrivial solutions is shown if either $\partial \Omega$ has non-negative mean curvature or $\Omega$ is an annulus, also for functions $f$  globally Lipschitz and $N\geq 2$. On the other hand, in \cite{MR937538} (see also \cite{Dancer}), a condition similar to $F(s)\leq 0$ is imposed, and the authors prove the nonexistence of positive solutions which satisfy a certain extra property; no geometric condition on $\Omega$ is assumed.

In the present paper, inspired by the results in \cite{MR937538}, we prove that there is no nontrivial solution to problem \eqref{eq10} provided $F(s)\leq 0$, being $f$ a  locally Lispchitz function. Here, no additional hypotheses on $\Omega$, $N$ nor $f$ are imposed. This exposes the unexpected fact that there is no geometric assumption on $\Omega$ that gives a nontrivial solution.


%
%
%
%

\section{main result}  

\begin{theorem}If $F(s)\leq 0$ for all $s\in\mathbb{R}$, there exists no nontrivial solution to \eqref{eq10}.
\end{theorem}

\begin{proof} Clearly, zero is a solution. We argue by contradiction and assume that there exists a nontrivial solution $u$ to \eqref{eq10}. First of all, notice that $-u$  is a solution to
\begin{equation*}
\begin{cases}
-\Delta u =-f(-u) &\text{in }\Omega,
\\
u=0 &\text{on }\partial \Omega.
\end{cases}
\end{equation*}
Since the function $-f(-s)$ is under the hypotheses of the theorem, there is no loss of generality in assuming that $u_\infty :=\max_{x\in \overline\Omega}u(x)>0$. On the other hand, since  $f$ is locally Lipschitz and the value of $f(s)$ for $s >u_\infty$ is irrelevant, we can also assume that $f$ is globally Lipschitz, with Lipschitz constant $L>0$, and that $\lim_{s\to +\infty} f(s)=-\infty$.

 It is easy to check that $f(u_\infty)>0$. Indeed, arguing by contradiction, assume that $f(u_\infty)\leq 0$. Then, 
\begin{equation}\label{eq3}
-\Delta u_\infty + L u_\infty\geq f(u_\infty)+Lu_\infty\quad\text{ in }\Omega.
\end{equation}
Moreover, we have proved that
\begin{equation}\label{eq2}
-\Delta u + Lu=f(u)+Lu\quad\text{ in }\Omega.
\end{equation}
Subtracting \eqref{eq2} from \eqref{eq3}, and using that $f(s)+Ls$ is non-decreasing, we obtain
\[-\Delta (u_\infty-u) + L (u_\infty-u)\geq f(u_\infty)+Lu_\infty-f(u)-Lu\geq 0\quad\text{ in }\Omega.\]
Since $u_\infty>u$ on $\partial\Omega$, the strong maximum principle implies that $u_\infty>u$ in $\Omega$, which is a contradiction.

Thus, the fact that $f(u_\infty)>0$ implies that there are $s_1,s_2>0$ such that $s_1<u_\infty<s_2$ and 
\begin{equation}\label{ineqf}
f(s)>0\quad\forall s\in (s_1,s_2).
\end{equation}
Moreover, since $F(s)\leq 0$ and $\lim_{s\to +\infty}f(s)=-\infty$, we can choose respectively $s_1$ and $s_2$ such that $f(s_1)=f(s_2)=0$.  Further, we can assume that $F(s_2)<0$ since, otherwise (i.e., if $F(s_2)=0$), we can modify $f$ to another $L$-Lipschitz function $f^*$ such that $f(s)>f^*(s)>0$ for $s\in (u_\infty,s_2)$ and $f=f^*$ elsewhere. In this way, $u$ is still a solution to \eqref{eq10}, but now $F(s_2)<0$.

Now we will find a family of supersolutions to \eqref{eq10} which will lead to a contradiction by comparison with $u$. For this purpose, we follow the original reasoning in \cite{MR937538}, which in principle is performed for $f\in\mathcal{C}^1(\mathbb{R})$. Here we adapt the proof to our setting and check that it also works for Lipschitz functions $f$.
  

  Indeed, consider the following initial value problem
\begin{equation*}
\begin{cases}
-w''(r)=f(w(r)),\quad\forall r>0,
\\
w(0)=s_2,
\\
w'(0)=-\sqrt{-F(s_2)}.
\end{cases}
\end{equation*}
Since $f$ is Lipschitz there is a unique solution $w\in \mathcal{C}^2([0,+\infty))$.  Multiplying the equation by $w'(r)$ and integrating, we obtain
\begin{align}\label{salvador}
\nonumber  (w'(r))^2 & = -F(s_2)+2\int_{w(r)}^{s_2}f(s)ds
\\
&=F(s_2)-2F(w(r)).
\end{align}
Thus, using \eqref{ineqf} we get that 
\begin{equation}\label{eq4}
(w'(r))^2>0\text{ for }w(r)\in [s_1,s_2].
\end{equation}
Now, since $w(0)=s_2$ and $w'(0)<0$, we deduce easily that $w(r)\in (s_1,s_2)$ for all $r>0$ small enough. We claim now that there exists $r_0>0$ such that $w(r_0)=s_1$. Indeed, assume by contradiction that $w(r)>s_1$ for all $r>0$. Then, by \eqref{eq4} we have that $w$ is decreasing in $(0,+\infty)$. Hence, there exists $s_3\in [s_1,s_2)$ such that $\lim_{r\to+\infty}w(s)=s_3$. But this is impossible as $w''(r)=-f(w(r))<0$ for all $r>0$, i.e., $w$ is concave.

In consequence, since $w(r_0)=s_1$ and $w'(r_0)<0$, we deduce that $\inf_{r\geq 0}w(r)<~s_1$. Moreover, it is easy to show that $\inf_{r\geq 0}w(r)>0$. Indeed, assuming otherwise, there exists a sequence $\{r_n\}\subset [0,+\infty)$ such that $\lim_{n\to\infty}w(r_n)=0$. Then, for $n$ large enough, we deduce from \eqref{salvador} that $(w'(r_n))^2<\frac{F(s_2)}{2}<0$, a contradiction. 

Thus, we have proved that
\begin{equation}\label{infimo}
0 < \inf w<s_1.
\end{equation}
Next, we define
\begin{equation*}
W(r)=
\left\{
\begin{array}{ll}
s_2, & r\in (-\infty,0\,],
\\
\min\{w(r),s_2\}, & r\in (0,\infty).
\end{array}
\right.
\end{equation*}
Since we can assume that $f(s)<0$ for $s>s_2$, it follows that $w$ is convex if $w(r)>s_2$. This implies that, if $w(r_2)=s_2$ for some $r_2>0$, then $W(r)=s_2$ for all $r\geq r_2$. Otherwise, $w(r)<s_2$ for all $r> 0$, so $W(r)=w(r)$ for all $r>0$.

  For every $t\in\mathbb{R}$, consider the family of parametric functions $v_t(x)=W(x_1-t)$ for all $x=(x_1,...,x_N)\in\mathbb{R}^N$. We will prove now that $u(x)\leq v_t(x)$ for all $x\in\overline\Omega$ and for all $t\in\mathbb{R}$ using the sweeping principle of Serrin. Indeed, let 
\[U=\{t\in\mathbb{R}:u(x)\leq v_t(x)\text{ for all }x\in\overline\Omega\}.\] 
Note that $v_t=s_2$ for $t$ large enough, and $u<s_2$ in $\overline\Omega$, so $U$ is nonempty. Notice also that $W$ is a globally Lipschitz function, so the function $t\mapsto v_t(x)$ is continuous uniformly in $x$. In particular, $U$ is closed. 

Let us now take $t\in U$. Observe that $v_t\in W^{1,\infty}(\Omega)$ and $-\Delta v_t\geq  f(v_t)$ in $\Omega$ (in the weak sense). Then, since $s\mapsto f(s)+Ls$ is non-decreasing and $u\leq v_t$ in $\overline\Omega$, we have that $-\Delta(v_t-u)+L(v_t-u)\geq 0$ in $\Omega$. Notice that 
\[u(x)=0<\inf w\leq v_t(x)\quad\forall x\in\partial\Omega,\]
so $v_t\not\equiv u$. Then, the strong maximum principle implies that $u(x)<v_t(x)$ for all $x\in\overline\Omega$. Therefore, the uniform continuity of $s\mapsto v_s$ implies that there exits $T>0$, independent of $x$, such that $u(x)<v_s(x)$ for all $x\in\overline\Omega$ and for all $s\in (t-T,t+T)$. That is to say, $(t-T,t+T)\subset U$, so $U$ is open. In conclusion, $U=\mathbb{R}$, and thus, $u\leq v_t$ for all $t\in\mathbb{R}$. In consequence,
\[u(x)\leq\inf_{t\in\mathbb{R}}v_t(x)=\inf_{r>0}w(r)<s_1,\quad\forall x\in\Omega,\]
which is a contradiction with the fact that $u_\infty\in (s_1,s_2)$.
\end{proof}

\section*{Acknowledgements}

First and second author are supported  by  MINECO-FEDER grant MTM2015-68210-P. First author is also supported by  Junta de Andaluc\'ia FQM-116 (Spain) and Programa de Contratos Predoctorales del Plan Propio de la Universidad de Granada.

\bibliographystyle{plainnat-linked-initials}
\bibliography{bibliography}

\end{document}